\RequirePackage{amsmath}
\newif\ifarxiv
\arxivtrue
\ifarxiv
\PassOptionsToClass{11pt}{llncs}
\fi
\documentclass[runningheads, envcountsame]{llncs}
\usepackage{amssymb,amsmath}
\usepackage{graphicx,enumerate}
\usepackage{etoolbox,paralist}
\usepackage{subfiles}
\usepackage{todonotes}
\usepackage{refcount}
\usepackage{algorithm}
\usepackage[noend]{algpseudocode}
\usepackage{color}
\usepackage{amsfonts,cite}
\providetoggle{long}
\usepackage[none]{hyphenat}
\settoggle{long}{true}
\renewcommand{\O}{\ensuremath{\mathcal{O}}}

\makeatletter
\newcommand\footnoteref[1]{\protected@xdef\@thefnmark{\ref{#1}}\@footnotemark}
\makeatother

\usepackage{thmtools,thm-restate}

\ifarxiv
\usepackage{fullpage}
\usepackage[appendix=inline]{apxproof} %
\else
\usepackage[appendix=strip]{apxproof} %
\fi
\newtheorem{fact}[theorem]{Fact}

\usepackage{tikz}
\usetikzlibrary{calc}
\usepackage{tkz-graph}
\tikzstyle{vertex}=[circle, draw, inner sep=0pt, minimum size=6pt]

\newcommand{\noinstance}{\texttt{no}-instance}

\title{Strong cliques in diamond-free graphs}

\author{Nina Chiarelli$^{1}$%\\
\and
Berenice Mart\'inez-Barona$^{2}$%\\
\and
Martin Milani\v c$^{1}$
\and \\
J\'er\^{o}me Monnot%\thanks{Deceased author.}
$^{3}$ %\\
\and
Peter Mur\v si\v c$^{1}$ %\\
}

\institute{
$^1${\small University of Primorska, Koper, Slovenia}\\%, FAMNIT, Glagolja\v ska 8, SI6000 Koper, Slovenia}\\
{\tt \{nina.chiarelli, martin.milanic, peter.mursic\}@famnit.upr.si}\\
$^2${\small Universitat Polit\`ecnica de Catalunya, Barcelona, Catalonia, Spain}\\
{\tt berenice.martinez@upc.edu}\\
$^3${\small LAMSADE, University Paris-Dauphine,
Paris Cedex 16, France}\\
}

\begin{document}
\pagestyle{plain}
\maketitle
\begin{sloppypar}
\begin{abstract}
A strong clique in a graph is a clique intersecting all inclusion-maximal stable sets. Strong cliques play an important role in the study of perfect graphs. We study strong cliques in the class of diamond-free graphs, from both structural and algorithmic points of view.
We show that the following five {\sf NP}-hard or {\sf co-NP}-hard problems remain intractable when restricted to the class of diamond-free graphs: Is a given clique strong? Does the graph have a strong clique? Is every vertex contained in a strong clique? Given a partition of the vertex set into cliques, is every clique in the partition strong? Can the vertex set be partitioned into strong cliques?

On the positive side, we show that the following two problems whose computational complexity is open in general can be solved in linear time in the class of diamond-free graphs: Is every maximal clique strong? Is every edge contained in a strong clique? These results are derived from a characterization of diamond-free graphs in which every maximal clique is strong, which also implies an improved
% strong variant of
Erd\H{o}s-Hajnal property for such graphs.
%{\color{red}the product of the clique and the stability numbers is at least the number of vertices}.
\end{abstract}
\end{sloppypar}

\begin{sloppypar}
\\
\end{sloppypar}

\section{Introduction}\label{sec:intro}

\paragraph{Background and motivation.}

\begin{sloppypar}
Given a graph $G$, a \emph{clique} in $G$ is a set of pairwise adjacent vertices, and a \emph{stable set} (or \emph{independent set}) is a set of pairwise non-adjacent vertices. A clique (resp., stable set) is \emph{maximal} if it is not contained in any larger clique (resp., stable set). A clique is \emph{strong} if intersects all maximal stable sets and a strong stable set is defined analogously. The concepts of strong cliques and strong stable sets in graphs play an important role in the study of perfect graphs (see, e.g.,~\cite{MR778744}) and were studied in a number of papers (see, e.g.,~\cite{HUJDUROVIC20192738,MR1301855,S0012365X0900547020100101, MR1344757,BOROS201478, S009589560400028020040101, MR3278773, S0166218X1000397520110101, MR1212874,BOROS201747, HMR2018,MR2080087, MR3623393,MR3474710,e6c81804d9c9450c946b6f3cece46881,alcon2019characterization}).
%Several algorithmic problems related to strong cliques and stable sets in graphs are
%{\sf NP}-hard
%or {\sf co-NP}-hard, in particular:
% In particular, this is the case for the following problems:
%\begin{itemize}
%\item {\sc Strong Clique}: Given a clique $C$ in a
%graph $G$, is $C$ strong? (which is {\sf co-NP}-complete, see~\cite{MR1344757}),
%\item {\sc Strong Clique Existence}: % \Pi_2^p?
%Given a graph $G$, does $G$ have
%a strong clique? (which is {\sf NP}-hard, see~\cite{MR1301855}),
%\item {\sc Strong Clique Vertex Cover}: % \Sigma_2^p?
%Given a graph $G$, is every vertex
%contained in a strong clique?
%\ifarxiv
%(which is {\sf co-NP}-hard, %see~\cite{HUJDUROVIC20192738}\footnote{\label{fn1}In~\cite{HUJDUROVIC20192738}, the authors %state that the
%problem is {\sf NP}-hard, but their proof actually shows {\sf co-NP}-hardness.}),
%\else
%(which is {\sf co-NP}-hard, see~\cite{HUJDUROVIC20192738}),
%\fi
%\item {\sc Strong Clique Partition}: % \Sigma_2^p?
%Given a graph $G$ and a partition of its vertex set into cliques, is every clique in the %partition strong?
%\ifarxiv
% (which is {\sf co-NP}-hard, see~\cite{HUJDUROVIC20192738}.\footnotemark[\ref{fn1}]),
%\else
%(which is {\sf co-NP}-hard, see~\cite{HUJDUROVIC20192738}),
%\fi
%item {\sc Strong Clique Partition Existence}: % \Sigma_2^p?
%Given a graph $G$, can its vertex set be partitioned into strong cliques?  (which is {\sf %co-NP}-hard, see~\cite{MR1217991,MR1161178,HMR2018}).
%\end{itemize}
Several algorithmic problems related to strong cliques and stable sets in graphs are
{\sf NP}-hard
or {\sf co-NP}-hard, in particular:
% In particular, this is the case for the following problems:
\begin{itemize}
\item {\sc Strong Clique}: Given a clique $C$ in a
graph $G$, is $C$ strong?
\item {\sc Strong Clique Existence}: % \Pi_2^p?
Given a graph $G$, does $G$ have
a strong clique?
\item {\sc Strong Clique Vertex Cover}: % \Sigma_2^p?
Given a graph $G$, is every vertex
contained in a strong clique?
\item {\sc Strong Clique Partition}: % \Sigma_2^p?
Given a graph $G$ and a partition of its vertex set into cliques, is every clique in the partition strong?
\item {\sc Strong Clique Partition Existence}: % \Sigma_2^p?
Given a graph $G$, can its vertex set be partitioned into strong cliques?
\end{itemize}
The first problem in the above list is {\sf co-NP}-complete (see~\cite{MR1344757}), the second one is {\sf NP}-hard (see~\cite{MR1301855}), and the remaining three are {\sf co-NP}-hard (see~\cite{HUJDUROVIC20192738} for the third and the fourth problem\footnote{In~\cite{HUJDUROVIC20192738}, the authors state that the two problems are {\sf NP}-hard, but their proof actually shows {\sf co-NP}-hardness.} and~\cite{MR1217991,MR1161178,HMR2018} for the fifth one).\end{sloppypar}

\begin{sloppypar}
Another interesting property related to strong cliques is the one defining CIS graphs. A graph is said to be \emph{CIS} if every maximal clique is strong, or equivalently, if every maximal stable set is strong, or equivalently, if every maximal clique intersects every maximal stable set. Although the name CIS (Cliques Intersect Stable sets) was suggested by Andrade et al.~in a recent book chapter~\cite{e6c81804d9c9450c946b6f3cece46881}, this concept has been studied under different names
since the 1990s~\cite{MR1344757,S0012365X0900547020100101,BOROS201478,MR3278773,
S009589560400028020040101,
S0166218X1000397520110101,
MR2234986} (see~\cite{MR1344757} for a historical overview).
Several other graph classes studied in the literature can be defined in terms of properties involving strong cliques (see, e.g., \cite{MR1212874,BOROS201747, HMR2018, HUJDUROVIC20192738, MR3623393}).
\end{sloppypar}

There are several intriguing open questions related to strong cliques and strong stable sets, % can be found in the literature,
for instance: (i) What is the complexity of determining whether every edge of a given graph is contained in a strong
clique? %, or, equivalently, of recognizing general partition graphs
(See, e.g.,\cite{MR3474710}.)
(ii) What is the complexity of recognizing CIS graphs?
(See, e.g.,\cite{e6c81804d9c9450c946b6f3cece46881}.)
(iii) Is there some $\varepsilon>0$ such that every $n$-vertex CIS graph has either a clique or a stable set of size at least $n^\varepsilon$?
(See~\cite{alcon2019characterization}.)

The main purpose of this paper is to study strong cliques in the class of diamond-free graphs. The \emph{diamond} is the graph obtained by removing an edge from the complete graph on four vertices, and a graph $G$ is said to be \emph{diamond-free} if no induced subgraph of $G$ is isomorphic to the diamond. Our motivation for focusing on the class of diamond-free graphs comes
from several sources. First, no two maximal cliques in a diamond-free graph share an edge, which makes interesting the question to what extent this structural restriction is helpful for understanding strong cliques.
Second, structural and algorithmic questions related to strong cliques in particular graph classes were extensively studied in the literature (see, e.g.,\cite{MR3474710,alcon2019characterization,HUJDUROVIC20192738,HMR2018,MR3278773,BOROS201478,MR2080087}), so this work represents a natural continuation of this line of research. Finally, this work furthers the knowledge about diamond-free graphs. In 1984, Tucker proved the Strong Perfect Graph Conjecture, now Strong Perfect Graph Theorem, for diamond-free graphs \cite{TUCKER1987313}, and more recently there has been an increased interest regarding the coloring problem and the chromatic number of diamond-free graphs and their subclasses (see, e.g., \cite{BRANDSTADT2012471, DABROWSKI2017410, Karthick:2018:CNP:3238799.3238823, KLOKS2009733}). Diamond-free graphs also played an important role in a recent work of Chudnovsky et al.~\cite{CHUDNOVSKY202045} who proved
that there are exactly $24$ 4-critical $P_6$-free graphs.

\begin{sloppypar}
\paragraph{Our contributions.}

Our study of strong cliques in diamond-free graphs is done from several interrelated points of view.
First, we give an efficiently testable characterization of diamond-free CIS graphs. A vertex $v$ in a graph $G$ is \emph{simplicial} if its closed neighborhood is a clique in $G$. Any such clique will be referred to as a \emph{simplicial clique}.
% A vertex $v$ in a graph $G$ is \emph{simplicial} if its neighborhood (or, equivalently, its closed neighborhood) is a clique in $G$. A clique $C$ in a graph $G$ is \emph{simplicial} if it consists of a simplicial vertex and all its neighbors.
We say that a graph $G$ is \emph{clique simplicial} if every maximal clique in $G$ is simplicial. The characterization is as follows.
\end{sloppypar}

\begin{restatable}[]{theorem}{thmMain}
\label{thm:main}
Let $G$ be a connected diamond-free graph. Then $G$ is CIS if and only if
$G$ is either clique simplicial, $G\cong K_{m,n}$ for some $m,n\ge 2$, or
$G\cong L(K_{n,n})$ for some $n\ge 3$.
\end{restatable}

Second, we derive several consequences of Theorem~\ref{thm:main}. A graph class $\mathcal{G}$ is said to satisfy the \emph{Erd\H{o}s-Hajnal property} if there exists some $\varepsilon>0$ such that every graph $G\in \mathcal{G}$ has either a clique or a stable set of size at least $|V(G)|^\varepsilon$. The well-known Erd\H{o}s-Hajnal Conjecture~\cite{MR1031262} asks whether for every graph $F$, the class of $F$-free graphs has the Erd\H{o}s-Hajnal property. The conjecture is still open, but it has been confirmed for graphs $F$ with at most $4$ vertices (see, e.g.,~\cite{MR3150572}).
In the case when $F$ is the diamond, a simple argument shows that the inequality holds with $\varepsilon = 1/3$ (see~\cite{MR1425208}), but it is not known whether this value is best possible.
Theorem~\ref{thm:main} implies the following improvement for the diamond-free CIS graphs.

\begin{restatable}[]{theorem}{thmEH}
\label{}
Let $G$ be a diamond-free CIS graph. Then $\alpha(G)\cdot \omega(G) \ge |V(G)|$.
Consequently, $G$ has either a clique or a stable set of size at least $|V(G)|^{1/2}$.
\end{restatable}

\begin{sloppypar}
Next, we develop a linear-time algorithm to test if every edge of a given diamond-free graph is in a simplicial clique. This leads to the following algorithmic consequence of Theorem~\ref{thm:main}.
\end{sloppypar}

\begin{restatable}[]{theorem}{thmRecognition}
\label{thm:diamond-free-recognition}
There is a linear-time algorithm that determines whether a given diamond-free graph $G$ is CIS.
\end{restatable}

Theorem~\ref{thm:diamond-free-recognition} implies a linear-time algorithm for testing if every edge of a given diamond-free graph is contained in a strong clique. Furthermore, as a consequence of Theorem~\ref{thm:diamond-free-recognition} and other results in the literature, we report on the following partial progress on the open question about the complexity of recognizing CIS graphs.

\begin{sloppypar}
\begin{restatable}[]{theorem}{corFfree}
\label{thm:F-free}
For every graph $F$ with at most $4$ vertices, it can be determined in polynomial time whether a given $F$-free graph is CIS.
%For every graph $F$ with at most $4$ vertices, there is a polynomial-time algorithm to determine whether a given $F$-free graph is CIS.
\end{restatable}
\end{sloppypar}

%Theorem~\ref{thm:diamond-free-recognition} also implies linear-time algorithms for testing any of the following four properties of a given diamond-free graph: general partition (that is, every edge is in a strong clique), strongly equistable, equistable, triangle (see Section~\ref{sec:equistable}). For each of these four properties, the computational complexity of testing the property for general graphs is an open question.

\begin{sloppypar}
Finally, we complement the above efficient characterizations with hardness results about several problems related to strong cliques when restricted to the class of diamond-free graphs. More specifically, using reductions from the {\sc $3$-Colorability} problem in the class of triangle-free graphs we show the following.
\end{sloppypar}

\begin{sloppypar}
\begin{restatable}[]{theorem}{thmHard}
\label{}
When restricted to the class of diamond-free graphs,
the {\sc Strong Clique}, {\sc Strong Clique Existence}, {\sc Strong Clique Vertex Cover}, and {\sc Strong Clique Partition} problems are {\sf co-NP}-complete, and the {\sc Strong Clique Partition Existence} problem is {\sf co-NP}-hard.
\end{restatable}
\end{sloppypar}

%(Theorems~\ref{thm:Strong-Clique}--\ref{thm:Strong-Clique-Partition-Existence}, respectively).
%(Theorems~\ref{thm:Strong-Clique},~\ref{thm:Strong-Clique-Existence},~\ref{thm:Strong-Clique-Vertex-Cover},~\ref{thm:Strong-Clique-Partition}, and~\ref{thm:Strong-Clique-Partition-Existence},
%respectively).

%\subsection{Structure of the paper}

%\begin{sloppypar}
%We collect the main notations, definitions, and some preliminary results in Section~\ref{sec:prelim}.
%We characterize diamond-free CIS graphs in Section~\ref{sec:diamond-free-CIS}
%and derive several consequences of this characterization
%in Section~\ref{sec:consequences}. This includes
%results on the existence of a large clique or a stable set in diamond-free CIS graphs (Section~\ref{sec:EH}),
%a linear-time algorithm for testing the CIS property in the class of diamond-free graphs (Section~\ref{sec:algo}),
%polynomial-time algorithms for testing the CIS property in any class of graphs excluding an induced subgraph with at most $4$ vertices
%(Section~\ref{sec:F-free}), and polynomial-time algorithms for recognizing diamond-free equistable graphs and related classes (Section~\ref{sec:equistable}). In Section~\ref{sec:hardness} we prove the hardness results.
%\end{sloppypar}

\ifarxiv\else
Due to space restrictions, most proofs are omitted and will appear in the full version of this work.
%deferred to the appendix.
\fi

\section{Preliminaries}\label{sec:prelim}

We consider only graphs that are finite and undirected. We refer to simple graphs as graphs and to graphs with multiple edges allowed as multigraphs. Let $G=(V,E)$ be a graph with vertex set $V(G)=V$ and edge set $E(G)=E$. For a subset of vertices $X\subseteq V(G)$, we will denote by $G[X]$ the subgraph of $G$ induced by $X$, that is, the graph with vertex set $X$ and edge set $\{\{u,v\}\mid\{u,v\}\in E(G);\; u,v\in X\}$. We denote the complete graph, the path, and the cycle graph of order $n$ by $K_n$, $P_n$, and $C_n$, respectively. The graph $K_3$ will be also referred as a \emph{triangle}. By $K_{m,n}$  we denote the complete bipartite graph with parts of size $m$ and $n$. The fact that a graph $G$ is isomorphic to a graph $H$  will be denoted by $G\cong H$. We say that $G$ is $H$-\emph{free} if no induced subgraph of $G$ is isomorphic to $H$.

The \emph{neighborhood} of a vertex $v$ in a graph $G$, denoted by $N_G(v)$ (or just $N(v)$ if the graph is clear from the context), is the set of vertices adjacent to $v$ in $G$. The cardinality of $N_G(v)$ is the \emph{degree} of $v$ in $G$, denoted by $d_G(v)$ (or simply $d(v)$). The \emph{closed neighborhood}, $N(v)\cup\{v\}$, is denoted by $N[v]$. Given a graph $G$ and a set $X\subseteq V(G)$, we denote by $N_G(X)$ the set of vertices in $V(G)\setminus X$ having a neighbor in $X$. The \emph{line graph} of a graph $G$, denoted with $L(G)$, is the graph with vertex set $E(G)$ and such that two vertices in $L(G)$ are adjacent if and only if their corresponding edges in $G$ have a vertex in common. A \emph{matching} in a graph $G$ is a set of pairwise disjoint edges. A matching is \emph{perfect} if every vertex of the graph is an endpoint of an edge in the matching. Given a graph $G$, we denote by $\alpha(G)$ the maximum size of a stable set in $G$ and by $\omega(G)$ the maximum size of a clique in $G$.

We first recall
a %some
basic property of CIS graphs (see, e.g.,~\cite{e6c81804d9c9450c946b6f3cece46881}). An induced $P_4$, $(a,b,c,d)$, in a graph $G$ is said to be \emph{settled} (in $G$) if $G$ contains a vertex $v$ adjacent to both $b$ and $c$ and non-adjacent to both $a$ and $d$.

\begin{proposition}\label{prop:CIS-properties}\mbox{}
%\begin{enumerate}[(i)]
%\item A graph is CIS if and only if its complement is CIS.\label{CIS:complement}
%\item A disconnected graph is CIS if and only if each of its connected component is CIS. \label{CIS:disconnected}
%\item
In every CIS graph each induced $P_4$ is settled.\label{CIS:settled}
%\end{enumerate}
\end{proposition}

\ifarxiv
We will also use the following characterization of diamond-free graphs (see, e.g.,~\cite{FELLOWS20112}).

\begin{lemma}\label{lem:maximal-cliques-diamond-free}
A graph $G$ is diamond-free if and only if every edge is in a unique maximal clique.
\end{lemma}
\fi

The following statements are consequences of Lemma~\ref{lem:maximal-cliques-diamond-free}.

\begin{corollary}\label{cor:maximal-cliques-diamond-free}
A diamond-free graph $G= (V,E)$ has at most $|E|$ maximal cliques.
\end{corollary}

\begin{corollary}
\label{vNotInMaxClique}
Let $G$ be a diamond-free graph, let $C$ be a maximal clique in $G$ and let $u\in V(G)\setminus C$.
Then $|N(u)\cap C|\leq 1$.
\end{corollary}

%We will also use the following simple characterization of diamond-free graphs.
%
%\begin{restatable}[]{lemma}{lemneighborhooddiamondfree}\label{lem:neighborhood-diamond-free}
%A graph $G$ is diamond-free if and only if for every vertex $v\in V(G)$, the subgraph of $G$ induced by $N(v)$ is a disjoint union of complete graphs.
%\end{restatable}
%
%\begin{proof}
%Note that a graph $G$ is diamond-free if and only if for every vertex $v\in V(G)$, the subgraph of $G$ induced by $N(v)$ is $P_3$-free. The lemma now follows from the well-known (and easy to see) fact that a graph $G$ is $P_3$-free if and only if it is a disjoint union of complete graphs.
%\qed\end{proof}

Given a clique $C$ in $G$ and a set $S\subseteq V(G)\setminus C$, we say that clique $C$ is \emph{dominated by $S$} if every vertex of $C$ has a neighbor in $S$. Using this notion, simplicial cliques can be characterized as follows.

\begin{lemma}\label{lem:simplicial-cliques}
A clique $C$ in a graph $G$ is not simplicial if and only $C$ is dominated by $V(G)\setminus C$.
\end{lemma}

\begin{proof}
Note that $C$ is simplicial if and only if there exists a vertex $v\in C$ such that $C = N[v]$.
However, since $C$ is a clique, we have $C\subseteq N[v]$ for all vertices $v\in C$.
It follows that $C$ is simplicial if and only if there exists a vertex $v\in C$ such that $N(v)\subseteq C$. Equivalently,
$C$ is not simplicial if and only every vertex in $C$ has a neighbor outside $C$, that is, $C$ is dominated by $V(G)\setminus C$.
\qed\end{proof}

\begin{sloppypar}
A similar characterization is known for strong cliques (see, e.g., the remarks following~\cite[Theorem 2.3]{HMR2018}).
\end{sloppypar}

\begin{lemma}\label{lem:strong-cliques}
A clique $C$ in a graph $G$ is not strong if and only if it is dominated by a stable set $S \subseteq V(G)\setminus C$.
\end{lemma}

Since every simplicial clique is strong (cf.~Lemmas~\ref{lem:simplicial-cliques} and~\ref{lem:strong-cliques}), we obtain the following.

\begin{fact}\label{obs:clique simplicial}
Every clique simplicial graph is CIS.
\end{fact}

For the algorithmic results in Section~\ref{sec:algo}, we will need the following lemma, which appears as an exercise in~\cite{MR2063679}.

\begin{lemma}\label{lemma:list-sort}
There is a linear-time algorithm that takes as input a multigraph $G = (V,E)$ given by adjacency lists and a linear ordering $\sigma = (v_1,\ldots, v_n)$ of its vertex set, and computes the adjacency lists of $G$ in which the neighbors of each vertex are sorted increasingly with respect to~$\sigma$.
\end{lemma}

\section{A characterization of diamond-free CIS graphs}\label{sec:diamond-free-CIS}

\begin{sloppypar}
%Since our goal in this section is to give a characterization of diamond-free CIS graphs
%and every induced $P_4$ in a CIS graph is settled (see Proposition~\ref{prop:CIS-properties}\eqref{CIS:settled}),
We first derive a property of diamond-free graphs in which every induced $P_4$ is settled.
\end{sloppypar}

\begin{restatable}[]{lemma}{EdgeMaxClique}
\label{EdgeMaxClique}
Let $G$ be a connected diamond-free graph in which every induced $P_4$ is settled.
Then, either $G$ is complete bipartite, or
%\item[$(ii)$]
each edge $e\in E(G)$ that is a maximal clique in $G$ is also
%it holds that $e$ is a maximal clique if and only if it is
a simplicial clique.
%one of the following holds.
%\begin{itemize}
%\item[$(i)$] $G$ is complete bipartite.
%\item[$(ii)$] For each edge $e\in E(G)$ it holds that $e$ is a maximal clique if and only if it is a simplicial clique.
%\end{itemize}
\end{restatable}

%
%\EdgeMaxClique*
\begin{proof}
 Suppose that $G$ is not complete bipartite. %It is clear that if an edge $e\in E(G)$ is a simplicial clique, then it is a maximal clique.
 Suppose that there is an edge $uv\in E(G)$ such that $\{u,v\}$ is a non-simplicial maximal clique in $G$. Then, by Lemma~\ref{lem:simplicial-cliques}, $N(v)\setminus\{u\}\neq\emptyset$ and $N(u)\setminus\{v\}\neq\emptyset$. If $x\in N(u)\setminus\{v\}$ and $y\in N(v)\setminus\{u\}$, then $xy\in E(G)$ otherwise the path $(x,u,v,y)$ is a non-settled induced $P_4$ in $G$, which is a contradiction. Moreover, since $G$ is diamond-free, for any two different vertices $x,x'\in N(u)\setminus\{v\}$ and two different vertices $y,y'\in N(v)\setminus\{u\}$ we have $xx',yy'\notin E(G)$. This, together with the fact that $N(u)\cap N(v)=\emptyset$ implies that $N(u)$ and $N(v)$ are stable sets. Then, neither $u$ nor $v$ are contained in a triangle. Hence, $G[N[u]\cup N[v]]$ is complete bipartite.
  Since $G$ is connected and not complete bipartite, we can assume without loss of generality that there is a vertex $w\in V(G)\setminus(N[u]\cup N[v])$ such that $wx\in E(G)$ for some $x\in N(u)\setminus\{v\}$. Then, $(v,u,x,w)$ is a non-settled induced $P_4$ in $G$, a contradiction.   %Therefore, $\{u,v\}$ is a simplicial clique in $G$.
\qed\end{proof}

\begin{sloppypar}
Next, we analyze the global properties of an induced subgraph $H$ in a connected diamond-free CIS graph $G$ such that
$H\cong L(K_{n,n})$ for some $n\geq3$.

\begin{restatable}[]{lemma}{lemgridinCISgraph}
\label{lem:grid-in-CIS-graph}
Let $G$ be a connected diamond-free CIS graph. Suppose that $G$ contains an induced subgraph $H$ such that $H\cong L(K_{n,n})$ for some $n\geq3$. Then, either all the $2n$ copies of $K_n$ in $H$ are maximal cliques in $G$ or none of them is a maximal clique in $G$. Furthermore, if all the $2n$ copies of $K_n$ are maximal cliques in $G$, then $G=H$.
\end{restatable}
\end{sloppypar}

\begin{proof}
Let $V(H)=\{(i,j)\mid 1\leq i\leq j\leq n\}$ and $E(H)=\{(i,j)(k,l)\mid \mbox{either }i=k\;\mbox{ or\; $j=l$},$ $\mbox{but not both}\}$. For each $i\in\{1,\dots,n\}$ let $R_i=\{(i,j)\mid 1\leq j\leq n\}$ and for each
$j\in\{1,\dots,n\}$, let $L_j=\{(i,j)\mid 1\leq i\leq n\}$. Hence, for each $i,j\in\{1,\dots, n\}$, we have $G[R_i]\cong G[L_j]\cong K_n$.
Observe that for any two $i,j\in\{1,\dots,n\}$ we have that $R_i\cap L_j=\{(i,j)\} $. Suppose that there is a copy of $K_n$ in $H$ which is a maximal clique in $G$, without loss of generality let $L_1$ be such a clique.
Now, we show that for each $i\in\{1,\dots,n\}$ we have that $R_i$ is also a maximal clique in $G$. For this, suppose the contrary and let, without loss of generality, $R_1$ be a non-maximal clique in $G$. Consider a maximal clique $R_1'$ in $G$ such that $R_1\subset R_1'$ and let $v\in R'_1\setminus R_1$. Note that $v \notin V(H)$. Observe that for any $i,j\in\{2,\dots,n\}$, if $(i,j)\in N_G(v)$ then $G[\{(1,j-1),(1,j),(i,j), v\}]$ is a diamond, a contradiction. Therefore, for all $j\in\{1,\dots,n\}$ we have $N_G(v)\cap L_j=\{(1,j)\}$, implying that $\{v,(2,2), (3,3),\dots, (n,n)\}$ is a stable set in $G- L_1$ dominating $L_1$, a contradiction by Lemma \ref{lem:strong-cliques}. Therefore, all the cliques $R_i$ are also maximal cliques in $G$. By symmetry, under the same argument, we conclude that for all $j\in\{1,\dots,n\}$ the cliques $L_j$ are also maximal cliques in $G$.

Now, suppose that for all $i,j\in\{1,\dots,n\}$, $R_i$ and $L_j$ are maximal cliques in $G$ and that $G\neq H$. Since $G$ is connected, there exists a vertex $u\in V(H)$ and a vertex $v\in V(G)\setminus V(H)$ such that $uv\in E(G)$. Without loss of generality let $u=(1,1)$. By Corollary~\ref{vNotInMaxClique}, we have that $N_G(v)\cap (R_1\cup L_1)=\{u\}$, and for all $i,j\in\{2,\dots, n\}$ we have $|N_G(v)\cap R_i|\leq1$ and $|N_G(v)\cap L_j|\leq 1$. Let $U=(N_G(v)\cap V(H))$. Hence, $U$ is a stable set in $(H -L_1) -R_1 \cong L(K_{n-1,n-1})$, which corresponds to a matching $M$ in $K_{n-1,n-1}$. By  K\" onig's Theorem \cite{Konig1916}, every $k$-regular bipartite graph has a perfect matching. Hence there exists $U' \subseteq V(H) \setminus (L_1 \cup R_1 \cup U)$, which corresponds to a perfect matching in $K_{n-1,n-1} -M$, implying $U'$ is a stable set in $G$ of size $n-1$. Hence, the set $\{v\} \cup U'$ is a stable set in $G - L_1$ dominating $L_1$, which is a contradiction by Lemma~\ref{lem:strong-cliques}.
\qed\end{proof}

We will also need the following technical lemma about diamond-free graphs.

\begin{restatable}[]{lemma}{lknnstablediagonal}
\label{lknn_stable_diagonal}
Let $G$ be a diamond-free graph. Let $C'=\{v_{1},v_{2},\ldots,v_{\ell}\}$ be a clique in $G$ and let $C_1,\ldots,C_\ell$ be maximal cliques in $G$ such that for every $i\in\{1,\ldots,\ell\}$ we have $C_i \cap C'=\{v_{i}\}$, $2\leq \ell=|C'|\leq |C_i|$ and  $C_i \cap C_j = \emptyset$ for any $j \neq i$. Then,  there exists a stable set $S\subseteq \bigcup_{i=1}^\ell C_i$ in $G-C'$ of size $\ell-1$ such that $|N_G(S)\cap C'|=|S|$. Furthermore, if $\max_{i=1}^\ell |C_i| > \ell$, then there exists a stable set $S\subseteq \bigcup_{i=1}^\ell C_i$ in $G-C'$
such that $|S| = \ell$ and $S$ dominates $C'$.
\end{restatable}

\begin{proof}
Without loss of generality let $|C_1| \leq |C_2| \leq \ldots \leq |C_{\ell}|$.
By Corollary~\ref{vNotInMaxClique} for any $i,j\in\{1,\ldots,\ell\}$ with $j\neq i$, each vertex from $C_i\setminus C'$ is adjacent to at most one other vertex in $C_j\setminus C'$.
Let $S_0=\emptyset$ be our starting stable set that we are going to build up.
In the $j$-th step ($1\leq j<\ell $), since $|(C_j\setminus C') \cap N_G(S_{j-1})| \leq |S_{j-1}|=j-1<\ell-1\leq|C_j|-1$, there is a vertex $u_j\in C_j\setminus C'$ such that $N_G(S_{j-1})\cap\{u_j\}=\emptyset$. Let $S_j=S_{j-1} \cup \{u_j\}$. Hence, $S_j\subseteq \bigcup_{i=1}^j C_i$ is a stable set in $G-C'$ of size $j$ dominating $\{v_1,\dots,v_j\}$. If $|C_{\ell}|=\ell$, then $S_{\ell-1}\subseteq \bigcup_{i=1}^{\ell-1} C_i$ is a stable set in $G-C'$ of size $\ell-1$ dominating $\{v_1,\dots,v_{\ell-1}\}$ and such that $|N_G(S_{\ell-1})\cap C'|=|S_{\ell-1}|$.
If $|C_{\ell}|>\ell$, then $|(C_\ell\setminus C') \cap N_G(S_{\ell-1})| \leq |S_{\ell-1}|=\ell-1 <|C_\ell|-1$, implying that there is a vertex $u_\ell\in(C_\ell\setminus C')$ such that $N_G(S_{\ell-1})\cap\{u_\ell\}=\emptyset$. Let $S_\ell=S_{\ell-1} \cup \{u_\ell\}$.
Hence, $S_\ell$ is a stable set of size $\ell$ in $G-C'$ dominating $C'$.
\qed\end{proof}

We can now characterize the connected diamond-free CIS graphs. The result shows that, apart from the clique simplicial diamond-free graphs, there are only two highly structured infinite families of connected diamond-free CIS graphs.

\thmMain*
%
%\begin{proofsketch}
%%{\color{red}Space permitting, we can add here a short proof sketch (one paragraph).}
%It is not difficult to see that each of the three conditions is sufficient for $G$ to be CIS. To show necessity, consider a connected diamond-free CIS graph $G$ and suppose for a contradiction that it does not satisfy any of the three conditions. Let $C=\{v_1,\ldots,v_k\}$ be a smallest maximal clique in $G$ that is not simplicial, and for each $i\in \{1,\ldots, k\}$, let $C_i\not= C$ be a maximal clique containing $v_i$. First observe that if $k=2$, then by Proposition~\ref{prop:CIS-properties} and Lemma~\ref{EdgeMaxClique} we have that $G$ is complete bipartite, a contradiction, thus $k\geq3$. Next, we show using Lemmas \ref{lem:strong-cliques} and  \ref{lknn_stable_diagonal} that all cliques $C_i$ have size $k$.
%Finally we show that either Lemma \ref{lem:grid-in-CIS-graph} implies $G\cong L(K_{k,k})$ or by using Lemma \ref{lknn_stable_diagonal} we can construct a stable set in $V(G)\setminus C$ dominating $C$, which leads to a contradiction by Lemma~\ref{lem:strong-cliques}.\qed
%\end{proofsketch}

\begin{proof}
Each of the three conditions is sufficient for $G$ to be CIS. Complete bipartite graphs are $P_4$-free, and every $P_4$-free graph is CIS (see~\cite{MR818265}). Graphs of the form $L(K_{n,n})$ are also CIS
(see, e.g.,~\cite[Proposition 3.3]{MR3278773}). By Fact~\ref{obs:clique simplicial}, every clique simplicial graph is CIS.

To prove that there are no other connected diamond-free CIS graphs, let $G$ be a connected diamond-free CIS graph such that $G$ is not clique simplicial, $G\ncong K_{m,n}$ (with $m,n\geq2$), and $G\ncong L(K_{n,n})$ (with $n\geq3$). Among all maximal non-simplicial cliques in $G$, pick one with the smallest size and name it $C$. Observe that, by Proposition~\ref{prop:CIS-properties} and Lemma~\ref{EdgeMaxClique}, $k\geq3$. Let $C=\{v_{1},\ldots,v_{k}\}$. For each vertex $v_{i}\in C$ fix a maximal clique $C_i$ containing $v_{i}$ such that $C_i\neq C$, and let $A_i=C_i\setminus\{v_{i}\}$. Hence, for each $i\in\{1,\dots,k\}$ we have $A_i\neq\emptyset$.
 Let $H$ be the subgraph of $G$ induced by $\bigcup_{j=1}^kC_j$. Note that for all $i$, if $C_i$ is simplicial in $H$, then $C_i=N_H(x_i)$ for some $x_i\in A_i$. Let $\ell$ denote the number of cliques among all $C_i$'s such that $|C_i|\geq k$. Without loss of generality let $C_1,\ldots,C_\ell$ be  such that $k\leq|C_1|\leq\ldots\leq|C_\ell|$. Then, by minimality of $C$, for every $\ell<i\leq k$ we have $C_i$ is a simplicial clique in $G$. For each $\ell<i\leq k$ let $u_i\in C_i$ be a simplicial vertex in $G$, hence $u_i\in A_i$.
 If $\ell=0$, then $S=\{u_i\mid 1\leq i\leq k\}$ is a stable set in $G-C$ dominating $C$, implying by Lemma~\ref{lem:strong-cliques} that $C$ is not strong, a contradiction. If $\ell=1$ then, $S=\{u_i\mid 2\leq i\leq k\}\cup\{u\}$ is a stable set in $G-C$ dominating $C$, for any $u\in A_1$, a contradiction.
 Hence, $\ell\geq2$.
 Let $C'=\{v_{1},v_{2},\dots,v_{\ell}\}$. Then, the cliques $C',C_1,\dots, C_\ell$ satisfy the hypothesis of Lemma~\ref{lknn_stable_diagonal}. Hence, if $\ell<k$, there is a stable set $S\subseteq\bigcup_{i=1}^{\ell}A_i$ in $G-C'$ dominating $C'$. Hence, $S\cup\{u_{\ell+1},\dots, u_k\}$ is a stable set in $G-C$ dominating $C$, a contradiction. Therefore, $\ell=k$. Moreover, if $|C_k|>k$ by Lemma~\ref{lknn_stable_diagonal} we have a stable set in $G-C$ of size $k$ dominating $C$, a contradiction. Hence, $|C_k|=k$ implying that $|C_i|=k$ for any $i\in\{1,\dots,k\}$.

Suppose there exist a vertex $v \in A_i$, $i \in \{1,\ldots,k\}$ with $N_H(v) \cap A_j=\emptyset$ for some $j \in \{1,\ldots,k\} \setminus \{i\}$. Without loss of generality let $i=1$ and $j=k$. Apply Lemma \ref{lknn_stable_diagonal} to the first $k-1$ $C_i$'s to construct a stable set $S\subseteq\left(\bigcup_{i=2}^{k-1}A_i\right)\cup\{v\}$ in $G-C$ of size $k-1$ in such a way that $S_1=\{v\}$, that is, $v$ is the first vertex selected in the construction of the stable set $S$. This way we get a stable $S$ disjoint from $A_k$. Moreover, by Corollary~\ref{vNotInMaxClique} and since $N_G(v) \cap A_k = \emptyset$ we have $|N_H(S) \cap A_j|\leq |S|-1=(k-1)-1=k-2$. This implies there exists a vertex $u \in A_j $ disjoint from $S$ that is not adjacent to $S$. Hence $S\cup\{u\}$ is a stable set in $G-C$ of size $k$ that dominates $C$, a contradiction.

Therefore, for each $i\in\{1,\dots,k\}$ and any vertex $v\in A_i$ we have $|N_H(v)\cap A_j|=1$ for any $j \in \{1,\ldots,k\} \setminus \{i\}$. Observe that if for any $i\in\{1,\dots,k\}$ and any $x\in A_i$ we have $N_G(x)\cap\left(\bigcup_{j\neq i}A_j\right)\cong K_{k-1}$, then $H\cong L(K_{k,k})$ implying, by Lemma~ \ref{lem:grid-in-CIS-graph}, that $G\cong L(K_{k,k})$, a contradiction.
Hence, there are three different numbers $i,j,j'\in\{1,\dots,k\}$  and three vertices $u,v,w$ such that $u\in A_i$, $v\in A_j$, $w\in A_{j'}$, $u,v\in N(w)$ and $uw\notin E(G)$.
By symmetry, let us assume that $i=1$, $j=2$ and $j'=k$. Use the proof of Lemma \ref{lknn_stable_diagonal} to generate a stable set $S$ of size $k-1$ in such a way that $S_1=\{u\}$, $S_2=\{u,v\}$ and $S_{k-1}\subseteq\bigcup _{i'=1}^{k-1}A_{i'}$.
Since $N_G(\{u,v\}) \cap A_k=\{w\}$ we must have $|N_G(S_{k-1}) \cap A_k|\leq |S_{k-1}|-1=(k-1)-1=k-2$. This implies that there exist a vertex $v' \in A_k$ that is not adjacent to $S_{k-1}$. Hence, $S_{k-1} \cup \{v'\}$ is a stable set in $G-C$ dominating $C$, a contradiction.
\qed\end{proof}

\section{Consequences}% of Theorem~\ref{thm:main}}
\label{sec:consequences}

We now discuss several consequences of Theorem~\ref{thm:main}.

\subsection{Large cliques or stable sets in diamond-free CIS graphs}\label{sec:EH}

Recall that every diamond-free graph $G$ has either a clique or a stable set of size at least
$|V(G)|^{1/3}$. For diamond-free CIS graphs, Theorem~\ref{thm:main} leads to an improvement of the exponent to $1/2$.

\begin{restatable}[]{lemma}{lemcliquesimplicialalphaomega}
\label{lem:clique simplicial-alpha-omega}
Let $G$ be a clique simplicial graph. Then $\alpha(G)\cdot \omega(G) \ge |V(G)|$.
\end{restatable}

%\vspace{-3mm}

%
%\lemcliquesimplicialalphaomega*
\begin{proof}
Let $C_1,\ldots, C_k$ be all the maximal cliques of $G$. Since $G$ is clique simplicial, each $C_i$ is a simplicial clique. Selecting one simplicial vertex from each $C_i$ gives a stable set $S$ of size $k$, and
since every vertex is in some simplicial clique, we infer that $|V(G)|\le \sum_{i = 1}^k|C_i|\le |S|\cdot \omega(G) \le \alpha(G)\cdot \omega(G)$.
\qed\end{proof}

\thmEH*

\begin{proof}
First we show that it suffices to prove the statement for connected graphs. Let $G$ be a disconnected diamond-free CIS graph, let $C$ be a component of $G$ and let $G' = G-V(C)$. Then $C$ and $G'$ are diamond-free CIS graphs,
and by induction on the number of components we may assume that
$\alpha(C)\cdot \omega(C) \ge |V(C)|$ and $\alpha(G')\cdot \omega(G') \ge |V(G')|$.
Since  $\alpha(G) = \alpha(C)+\alpha(G')$ and $\omega(G) = \max\{\omega(C),\omega(G')\}$,
we obtain $|V(G)| = |V(C)|+|V(G')|\le \alpha(C)\cdot \omega(G)+ \alpha(G')\cdot \omega(G) = \alpha(G)\cdot \omega(G)$.

Now let $G$ be a connected diamond-free CIS graph. By Theorem~\ref{thm:main}, $G$ is either clique simplicial, $G\cong K_{m,n}$ for some $m,n\ge 2$, or $G\cong L(K_{n,n})$ for some $n\ge 3$. If $G$ is clique simplicial, then Lemma~\ref{lem:clique simplicial-alpha-omega} yields $|V(G)|\le \alpha(G)\cdot \omega(G)$. If $G\cong K_{m,n}$ for some $m,n\ge 2$, then $|V(G)| = m+n\le 2\cdot\max\{m,n\} = \omega(G)\cdot \alpha(G)$. Finally, if $G\cong L(K_{n,n})$, then $\alpha(G)$ equals the maximum size of a matching in $K_{n,n}$, that is, $\alpha(G) = n$,
and $\omega(G)$ equals the maximum degree of a vertex in $K_{n,n}$, that is, $\omega(G) = n$.
Thus, in this case equality holds, $|V(G)| = n^2 = \alpha(G)\cdot \omega(G)$.
\qed\end{proof}

\subsection{Testing the CIS property in the class of diamond-free graphs}\label{sec:algo}

\begin{sloppypar}
Our next consequence is a linear-time algorithm for testing the CIS property in the class of diamond-free graphs.
The bottleneck to achieve linearity is the recognition of the clique simplicial property.
Instead of checking this property directly, we check whether the graph is \emph{edge simplicial}, that is, every edge is contained in a simplicial clique.
Clearly, every clique simplicial graph is edge simplicial.
While the converse implication fails in general, the two properties are equivalent in the case of diamond-free graphs, where every edge is in a unique maximal clique%
\ifarxiv
(cf.~Lemma~\ref{lem:maximal-cliques-diamond-free})
\fi
.
\end{sloppypar}

\ifarxiv
\begin{fact}\label{obs:cs-es}
A diamond-free graph is clique simplicial if and only if it is edge simplicial.
\end{fact}
\fi

Recognizing if a general graph $G = (V,E)$ is edge simplicial can be done in time $\O(|V|\cdot|E|)$ (see~\cite{MR2302164,MR727924}).
The algorithm is based on the observation that within a simplicial clique, every vertex of minimum degree is a simplicial vertex (see~\cite{MR992923}). We show that in the case of diamond-free graphs, the running time can be improved to $\O(|V|+|E|)$.

\ifarxiv
To describe the algorithm, we need to introduce some terminology.
\fi
Given a graph $G$ and a linear ordering $\sigma=(v_1,\ldots, v_n)$ of its vertices, the $\sigma$\textit{-greedy stable set} is the stable set $S$ in $G$ computed by repeatedly adding to the initially empty set $S$ the smallest $\sigma$-indexed vertex as long as the resulting set is still stable. A \textit{degree-greedy stable set} is any $\sigma$-greedy stable set where $\sigma=(v_1,\ldots, v_n)$ satisfies $d(v_i)\leq d(v_j)$ for $i<j$. Note that a degree-greedy stable set does not need to coincide with a stable set computed by the greedy algorithm that iteratively selects a minimum degree vertex and deletes the vertex and all its neighbors: to compute a degree-greedy stable set, the vertex degrees are only considered in the original graph $G$, and not in the subgraphs obtained by deleting the already selected vertices and their neighbors.

Our first lemma analyzes the structure of a degree-greedy stable set in a graph in which the simplicial cliques cover all the vertices. In particular, it applies to edge simplicial graphs.

\begin{restatable}[]{lemma}{lemS}
\label{lem:S}
Let $G$ be a graph in which every vertex belongs to a simplicial clique and let $S$ be a degree-greedy stable set in $G$.
Then, $S$ consists of simplicial vertices only, one from each simplicial clique.
%every vertex in $S$ is simplicial and $S$ contains exactly one vertex from every simplicial clique in $G$.
\end{restatable}

%
%\lemS*
\begin{proof}
Let $\sigma=(v_1,\ldots, v_n)$ be a linear ordering of $V(G)$ with  $d(v_i)\leq d(v_j)$ if $i<j$ and such that $S$ is the $\sigma$-greedy stable set. Let $C$ be a simplicial clique in $G$ and let $v_i$ be the simplicial vertex in $C$ with the smallest index.
Then, all neighbors of $v_i$ have indices larger than $i$ and therefore $v_i$ is selected to be in $S$.
This shows that $S$ contains a vertex from $C$. Clearly, $S$ cannot contain two vertices from $C$ since $S$
is stable and $C$ is a clique. Thus, $S$ contains exactly one vertex from every simplicial clique in $G$.
Finally, suppose that $S$ contains some non-simplicial vertex $v$. By assumption on $G$, vertex $v$ belongs to some simplicial clique $C$. Let $w$ be the vertex from $C$ contained in $S$. Then $w$ is simplicial and thus $w\neq v$. This means that $S$ contains two adjacent vertices $v$ and $w$, which is in contradiction with the fact that $S$ is a stable set. Thus, every vertex in $S$ is simplicial.
\qed\end{proof}

Using Lemma~\ref{lem:S} we now show the importance of degree-greedy stable sets for testing if a given diamond-free graph is edge simplicial. The characterization is based on the following auxiliary construction. Given a graph $G=(V,E)$ and a stable set $S$ in $G$, we define the multigraph $G_S$ where $V(G_S)=V\setminus S$ and the multiset of edges is given by $$E(G_S)=\bigcup_{v\in S}\{xy\mid x\neq y \text{ and } x,y \in N_G(v)\}\,.$$
Note that since $S$ is a stable set in $G$, every edge in $E(G_S)$ indeed has both endpoints in $V\setminus S = V(G_S)$.

\begin{restatable}[]{lemma}{lemmaGSGS}
\label{lemma:G-S=G_S}
Let $S$ be a degree-greedy stable set in a diamond-free graph $G$. Then $G$ is edge simplicial if and only if $G_S=G-S$.
\end{restatable}

\begin{proof}
Suppose $G$ is edge simplicial. Each edge is therefore part of at least one simplicial clique and, since $G$ is diamond-free, each edge is in fact in exactly one simplicial clique. By Lemma~\ref{lem:S}, $S$ consists of simplicial vertices only, one from each simplicial clique. We need to show that the multigraphs $G_S$ and $G-S$ are equal.
The vertex sets $V(G_S)$ and $V(G-S)$ are equal by definition.
Since $G$ is a simple graph, so is $G-S$. Suppose that $G_S$ has a duplicate edge $xy$. Then there must exist two distinct vertices $v_1,v_2 \in S$ such that $\{v_1x,v_1y,v_2x,v_2y\} \subseteq E(G)$. Since $S$ is a stable set we have $v_1v_2\not\in E(G)$ and since $v_1$ is a simplicial vertex in $G$, vertices $x$ and $y$ are adjacent in $G$. But now $G$ contains an induced diamond on the vertices
$\{v_1,v_2,x,y\}$, a contradiction. Thus, $G_S$ is a simple graph.
Consider an edge $xy \in E(G-S)$. Since $G$ is edge simplicial, Lemma~\ref{lem:S} implies the existence of a vertex $v\in S$ such that $\{x,y,v\}$ is a clique in $G$, implying that $x,y\in N_G(v)$ and thus $xy\in E(G_S)$. Therefore $E(G-S)\subseteq E(G_S)$. Consider now an edge $xy\in E(G_S)$. By the definition of $G_S$, there exists a vertex $v\in S$ such that $\{x,y\}\subseteq N_G(v)$. Since $v\in S$, vertex $v$ is simplicial in $G$ and hence $xy\in E(G-S)$. Therefore $E(G-S)\supseteq E(G_S)$ and $G-S=G_S$.

Now we show the reverse implication by contraposition. Suppose that $G$ is not edge simplicial. Then there exists an edge $xy\in E(G)$ that is not contained in any simplicial clique of $G$. In particular, since every vertex in $S$ is simplicial in $G$, vertices $x$ and $y$ cannot belong to $S$. Thus $xy\in E(G-S)$. Furthermore, there does not exist a vertex $s \in S$ such that $x,y \in N_G(s)$. Hence the edge $xy$ is not in $E(G_S)$ but it is in $E(G-S)$. It follows that $G_S\neq G-S$.
\qed\end{proof}

The condition given by Lemma~\ref{lemma:G-S=G_S} can be tested in linear time%
\ifarxiv
.
\else
, see Algorithm~\ref{alg:DFESR2}. Together with Theorem~\ref{thm:main}, this leads to the following.

\thmRecognition*
\fi

\ifarxiv
\begin{restatable}[]{proposition}{propdiamondfreerecognition}
\label{prop:diamond-free-recognition2}
%There is a linear-time algorithm to
Algorithm~\ref{alg:DFESR2} runs in time $\mathcal{O}(|V|+|E|)$ and correctly determines whether a given diamond-free graph $G=(V,E)$ is edge simplicial.
\end{restatable}
\fi

\begin{algorithm}[!h]
\caption{Diamond-Free Edge Simplicial Recognition}\label{alg:DFESR2}
 \hspace*{\algorithmicindent} \textbf{Input:} A diamond-free graph $G$ given by adjacency lists $(L_v:v\in V(G))$.\\% of the vertices $v \in V$.\\ % V(G)$.\\
 \hspace*{\algorithmicindent} \textbf{Output:} ``Yes'' if $G$ is edge simplicial, ``No'' otherwise.
\begin{algorithmic}[1]
 \State compute a linear order $\sigma = (v_1,\ldots, v_n)$ of the vertices of $G$ such that \hbox{ $d(v_1)\leq \ldots \leq d(v_n)$};
 \State sort the adjacency list of each vertex in increasing order with respect to $\sigma$;
 \State set $S = \emptyset$, all the vertices are unmarked;
 \For{$i=1,\ldots, n$}
     \If {$v_i$ is not marked}
         \State mark all vertices in $N(v_i)$, $S=S\cup\{v_i\}$;
     \EndIf
 \EndFor
 \State compute the adjacency lists $L^-_w$ of $G-S$ based on the order of $V\setminus S$ induced by $\sigma$;
 \If {$\sum_{v \in S}{d(v)+1 \choose 2}>|E(G)|$}
     \State \Return ``No'';
 \EndIf
 \For{$w \in V\setminus S$}
     \State $L'_w=\text{empty list}$;
 \EndFor
 \For{$v \in S$}
     \For{$w \in N_G(v)$}
         \State append each element of $N_G(v)\setminus\{w\}$ at the end of $L'_w$;
     \EndFor
 \EndFor
 \For{$w \in V\setminus S$}
 		\If{$\text{length}(L'_w)\neq \text{length}(L^-_w)$}
 				\State \Return  ``No'';
 		\EndIf		
 \EndFor

 \State sort the adjacency lists $L'_w$ based on the linear order of $V\setminus S$ induced by $\sigma$;
 \For{$w \in V\setminus S$}
		 \If{$L'_w\neq L^-_w$}
 				\State \Return  ``No'';
		 \EndIf
 \EndFor
 \State \Return ``Yes'';
 \end{algorithmic}
 \end{algorithm}
\ifarxiv

%\propdiamondfreerecognition*
\begin{proof}
Let us first describe the algorithm informally. In lines 1--6, Algorithm~\ref{alg:DFESR2} computes a degree-greedy stable set $S$ of $G$.
In line 7, it computes the adjacency lists $L_w^-$ of the graph $G-S$.
In lines 8--9 it checks that a necessary condition for $G$ to be edge simplicial holds. (This check is only needed to achieve a linear running time.) In lines 10--14, the algorithm computes the adjacency lists $L_w'$ of the multigraph $G_S$.
Finally, in lines 15--21, the algorithm tests if $G_S = G-S$.

\begin{sloppypar}
We prove that the algorithm is correct by analyzing each of the possible outcomes.
Suppose first that the algorithm returned ``No'' at line 9. Then $\sum_{v\in S}{d(v)+1 \choose 2} > |E(G)|$. Suppose for a contradiction that $G$ is edge simplicial. By Lemma~\ref{lem:S}, $S$ consists of simplicial vertices only, one from each simplicial clique. Since $G$ is diamond-free, Lemma~\ref{lem:maximal-cliques-diamond-free} implies that every edge is in a unique maximal clique. Thus, for two different vertices $v,w\in S$, the cliques $N[v]$ and $N[w]$ are edge-disjoint. In particular, this implies that the value of $\sum_{v \in S}{d(v)+1 \choose 2}$, which counts the total number of edges
in these simplicial cliques, cannot exceed $|E(G)|$.
Hence, if $\sum_{v \in S}{d(v)+1 \choose 2} > |E(G)|$, then $G$ is indeed not edge simplicial.
Suppose next that the algorithm returned ``No'' at line 17 or 21. Then, the graphs $G_S$ and $G-S$ are not the same and by Lemma~\ref{lemma:G-S=G_S}, we can say that $G$ is not edge simplicial.
Finally, if the algorithm returns ``Yes'', then $G_S = G-S$ and thus $G$ is edge simplicial
by Lemma~\ref{lemma:G-S=G_S}.
\end{sloppypar}

Next we analyze the time complexity of the algorithm. We denote by $m$ the number of edges of $G$. Sorting the vertices with respect to their degrees in line 1 can be done in time $\O(n+m)$ using counting sort.
A further sorting of the adjacency lists in line 2 can be done in $\O(n+m)$ using Lemma \ref{lemma:list-sort}.
Line 3 takes $\O(n)$ time.
Lines 4--6 take altogether $\O(n+m)$ time. Line 7 takes $\O(n+m)$ time.
Lines 8--9 take $\O(|S|) = \O(n)$ time.
Lines 10--11 take $\O(n)$ time.
Lines 12--14 take time proportional to $\sum_{v\in S}(d(v)(d(v)-1)) \le 2\cdot \sum_{v\in S}{d(v)+1\choose 2}$, which is in $\O(m)$ since otherwise the algorithm would have returned ``No'' at line 9.
Lines 15--17 take $\O(n)$ time.
Since the algorithm performs line 18 only if
$\text{length}(L'_w) = \text{length}(L^-_w)$ for all $w\in V\setminus S$,
sorting of the adjacency lists in line 18 can be done using Lemma \ref{lemma:list-sort}
in time proportional to
$\O(n+\sum_{w\in V\setminus S}|L'_w|) =
\O(n+\sum_{w\in V\setminus S}|L^-_w|) = \O(n+m)$.
Similarly, the time needed to perform a check in line 20 is bounded by $\O(d_{G-S}(w)) = \O(d_G(w))$ for each iteration,
hence the total time needed for lines 19--21 is again $\O(n+m)$. Line 22 takes $\O(1)$ time.
We conclude that altogether, the algorithm runs in linear time.
\qed\end{proof}

\fi
%Proposition~\ref{prop:diamond-free-recognition2} and Fact~\ref{obs:cs-es} imply the following.
%
%\begin{corollary}\label{cor:algorithm}
%Algorithm~\ref{alg:DFESR2} runs in time $\mathcal{O}(|V|+|E|)$ and correctly determines whether a given diamond-free graph $G =(V,E)$ is clique simplicial.
%\end{corollary}

\ifarxiv
Theorem~\ref{thm:main}, Fact~\ref{obs:cs-es}, and
% and Corollary~\ref{cor:algorithm} imply the following result.
Proposition~\ref{prop:diamond-free-recognition2}
imply the following.

\thmRecognition*
\fi

%We now have everything ready to prove Theorem~\ref{thm:diamond-free-recognition}.

\begin{proof}
Let $G=(V,E)$ be a diamond-free graph. Since a graph is CIS if and only if all of its components are CIS (see, e.g.,~\cite{e6c81804d9c9450c946b6f3cece46881}),
% (see Proposition~\ref{prop:CIS-properties}\eqref{CIS:disconnected}),
we may assume that $G$ is connected.

By Theorem~\ref{thm:main}, it suffices to check if $G$ is isomorphic to either $K_{m,n}$ for some $m,n\ge 2$, or $L(K_{n,n})$ for some $n\ge 3$, or if $G$ is clique simplicial. We can check if $G$ is a complete bipartite graph in linear time as follows. We run breadth-first search from any fixed vertex $v$ to test if $G$ is bipartite and to compute the sets $X$ and $Y$ of vertices at distance one and two from $v$, respectively. Assuming $G$ is bipartite, we only need to check if $|V| = 1+|X|+|Y|$, and verify that the vertex degrees are what they should be (namely, $d_G(x) = |Y|+1$ for all $x\in X$ and $d_G(y) = |X|$ for all $y\in Y$).

To check if $G$ is isomorphic to the line graph of $K_{n,n}$ for some $n\ge 3$, we can first check in linear time whether $G$ is a line graph, using the algorithm of Roussopoulos~\cite{MR0424435}  or Lehot~\cite{MR0347690}. If $G$ is indeed a line graph, the algorithm will also compute a graph $H$ such that $G = L(H)$. As shown by Whitney~\cite{MR1506881}, if such a graph $H$ exists, then $H$ is unique, provided that $G$ is connected and has at least four vertices (which we may assume). It remains to check if $H\cong K_{n,n}$ for some $n\ge 3$, which can also be done in linear time, as explained above.

Finally, to check if $G$ is clique simplicial in linear time, we use Proposition~\ref{prop:diamond-free-recognition2} and Fact~\ref{obs:cs-es}.
% Corollary~\ref{cor:algorithm}.
\qed\end{proof}

A graph $G=(V,E)$ is \emph{general partition} if there exists a set $U$ and an assignment of vertices $x\in V$ to sets $U_x\subseteq U$ such that $xy\in E$ if and only if $U_x\cap U_y \neq\emptyset$ and for each maximal stable set $S$ in $G$, the sets $U_x$, $x\in S$ form a partition of $U$.
It is known that $G$ is a general partition graph if and only if every edge of $G$ is contained in a strong clique (see~\cite{MR1212874}).
Clearly, every CIS graph is a general partition graph, and the two properties are equivalent in the class of diamond-free graphs.
% Clearly, every CIS graph is a general partition graph, but not vice-versa, as shown by the $3$-sun.
% However, in the class of diamond-free graphs (and even in the class of $3$-sun-free graphs), these two properties are equivalent, as can be seen from the proof \hbox{of~\cite[Theorem 1]{MR1477536}} (as well as from~\cite[Lemma 1]{MR3474710}).
Therefore, Theorem~\ref{thm:diamond-free-recognition} implies the following.

\begin{corollary}
There is a linear-time algorithm that determines whether every edge of a given diamond-free graph $G$ is in a strong clique.
\end{corollary}

\subsection{Testing the CIS property in classes of $F$-free graphs}\label{sec:F-free}

\begin{sloppypar}
No good characterization or recognition algorithm for CIS graphs is known. Recognizing CIS graphs is believed to be {\sf co-NP}-complete~\cite{MR1344757}, conjectured to be {\sf co-NP}-complete~\cite{MR2234986}, and conjectured to be polynomial~\cite{e6c81804d9c9450c946b6f3cece46881}. Using Theorem~\ref{thm:diamond-free-recognition} together with some known results from the literature (on strong cliques, resp.~CIS graphs~\cite{HMR2018,alcon2019characterization,e6c81804d9c9450c946b6f3cece46881} along with~\cite{alekseev1991number,MR984569,MR1405872,MR476582}) implies
 %we now derive
that the CIS property can be recognized in polynomial time in any class of $F$-free graphs where $F$ has at most $4$ vertices.
\end{sloppypar}

\corFfree*

\begin{sloppypar}

\begin{proof}
Let $F$ be a graph with at most $4$ vertices. Clearly, it suffices to consider the case when $F$ has exactly four vertices. Furthermore,
since a graph is CIS if and only if its complement is CIS%(see, e.g.,~\cite{e6c81804d9c9450c946b6f3cece46881})
%by Proposition~\ref{prop:CIS-properties}\eqref{CIS:complement}
, if there is a polynomial-time algorithm to determine whether a given $F$-free graph is CIS, then there is also a polynomial-time algorithm to determine whether a given $\overline{F}$-free graph is CIS.
Thus, since there are $11$ graphs with $4$ vertices, $5$ pairs of complementary graphs $\{F$, $\overline{F}\}$, and one self-complementary graph, the $P_4$, we consider $6$ cases.

First, consider the case when $F\in \{K_4$, $\overline{K_4}\}$. If $G$ is a $K_4$-free graph, then $G$ has $\O(|V(G)|^3)$ maximal cliques, and hence one can test if $G$ is CIS simply by enumerating all the maximal cliques and testing, for each of them, if it is strong. By Lemma~\ref{lem:strong-cliques}, a clique $C$ in a graph $G$ is strong if and only if there is no stable set $S \subseteq V(G)\setminus C$ such that $|S|\le |C|$ and every vertex in $C$ has a neighbor in $S$. Since we only need to check this for cliques of constant size, we obtain a polynomial-time algorithm.

Second, suppose that $F$ is the diamond or its complement. In this case, Theorem~\ref{thm:diamond-free-recognition} applies.

Third, suppose that $F$ is the $C_4$ or its complement. A clique in a $C_4$-free graph is strong if and only if it is simplicial (see~\cite{HMR2018}), and hence a $C_4$-free graph is CIS if and only if it is clique simplicial. Every $C_4$-free graph $G$ has at most $\O(|V(G)|^2)$ maximal cliques (see~\cite{alekseev1991number,MR984569,MR1405872}, and they can be enumerated in polynomial time (e.g., by applying the algorithm of Tsukiyama et al.~\cite{MR476582} to the complement of $G$). Thus, we can test
in polynomial time if a given $C_4$-free graph is CIS by enumerating its maximal cliques and checking if they are all simplicial.

Fourth, suppose that $F$ is the paw (that is, a graph obtained from $K_3$ by adding to it a new vertex of degree $1$) or its complement.
By Proposition~\ref{prop:CIS-properties}, in every CIS graph each induced $P_4$ is settled. Thus, if a CIS graph $P_4$,  $(a,b,c,d)$, then it also contains an induced paw, on the vertex set $\{a,b,c,v\}$, where $v$ is any vertex that settles the $P_4$ (that is, $v$ is adjacent to both $b$ and $c$ and non-adjacent to both $a$ and $d$). It follows that a paw-free graph is CIS if and only if it is $P_4$-free. A polynomial-time recognition algorithm to determine whether a given paw-free graph is CIS follows.

Fifth, suppose that $F$ is the claw (that is, the complete bipartite graph $K_{1,3}$) or its complement.
A polynomial-time algorithm for recognizing claw-free CIS graphs was given in~\cite{alcon2019characterization}.

Finally, suppose that $F$ is the $P_4$. In this case, the recognition algorithm is trivial, since every $P_4$-free graph is CIS (see, e.g.,\cite{e6c81804d9c9450c946b6f3cece46881}).
\qed\end{proof}
\end{sloppypar}

\section{Hardness results}\label{sec:hardness}

\begin{sloppypar}
We consider five more decision problems related to strong cliques: {\sc Strong Clique}, {\sc Strong Clique Existence}, {\sc Strong Clique Vertex Cover}, {\sc Strong Clique Partition}, and {\sc Strong Clique Partition Existence} (see Section~\ref{sec:intro} for definitions). These problems were studied by Hujdurovi\'c et al.~in~\cite{HUJDUROVIC20192738}, who determined the computational complexity of these problems in the classes of chordal graphs, weakly chordal graphs, line graphs and their complements, and graphs of maximum degree at most three.
\end{sloppypar}

\begin{sloppypar}
In contrast with the problems of verifying whether every maximal clique is strong, or whether every edge is in a strong clique, we prove that all the above five problems are {\sf co-NP}-hard in the class of diamond-free graphs.
The hardness proofs are obtained using a reduction from the {\sc $3$-Colorability} problem in the class of triangle-free graphs:
Given a triangle-free graph $G$, can $V(G)$ be partitioned into three stable sets?
As shown by Kami\'{n}ski and Lozin in~\cite{MR2291884}, this problem is {\sf NP}-complete.
Furthermore, it is clear that the problem remains {\sf NP}-complete if we additionally assume that the input graph has at least five vertices and minimum degree at least three. Let $\mathcal{G}$ denote the class of all triangle-free graphs with at least five vertices and minimum degree at least $3$.
\end{sloppypar}

\begin{theorem}\label{thm:3-col}
The {\sc $3$-Colorability} problem is {\sf NP}-complete in the class $\mathcal{G}$.
\end{theorem}

The reductions are based on the following construction. Given a graph $G\in \mathcal{G}$, we associate to it two diamond-free graphs $G'$ and $G''$, defined as follows. The vertex set of $G'$ is $V(G)\times\{0,1,2,3\}$.
For every $v\in V(G)$, the set of vertices of $G'$ with value $v$ in the first coordinate forms a clique (of size $4$); we will refer to this clique as $C_v$. The set of vertices of $G'$ with value $0$ in the second coordinate forms a clique $C$ (of size $|V(G)|$).
For every $i\in \{1,2,3\}$ and every two distinct vertices $u,v\in V(G)$, vertices $(u,i)$ and $(v,i)$ are adjacent in $G'$ if and only if $u$ and $v$ are adjacent in $G$. There are no other edges in $G'$. The graph $G''$ is obtained from the graph $G'$ by adding, for each vertex $w\in V(G')\setminus C$, a new vertex $w'$ adjacent only to $w$. See Figure~\ref{fig:reduction} for an example.

\begin{figure}[h]
  \centering
  \includegraphics[width=\linewidth]{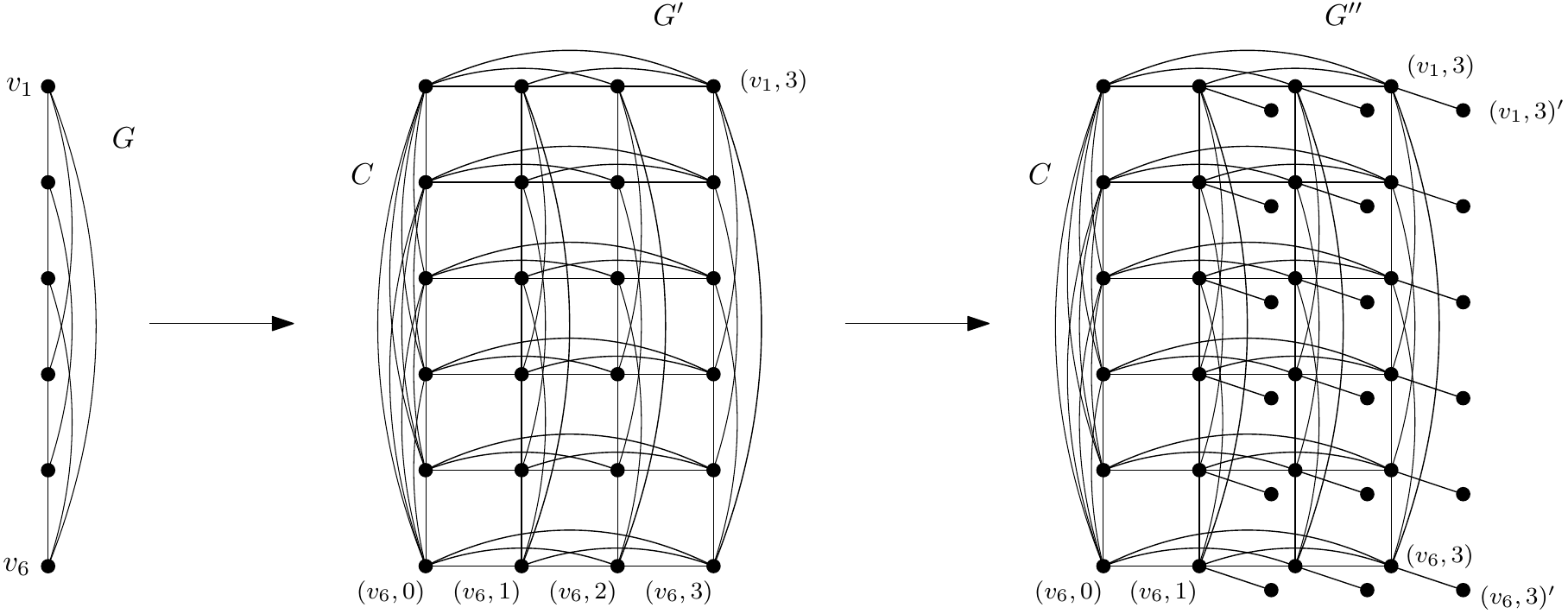}
  \caption{Transforming $G$ into $G'$ and $G''$.}
  \label{fig:reduction}
\end{figure}

\begin{restatable}[]{proposition}{propreductions}\label{prop:reductions}
Let $G\in \mathcal{G}$ and let $G'$ and $G''$ be the graphs constructed from $G$ as described above.
Then, $G'$ and $G''$ are diamond-free and the following statements are equivalent.
\begin{enumerate}
  \item $G$ is not $3$-colorable.
  \item $C$ is a strong clique in $G'$.
  \item $G'$ has a strong clique.
  \item $C$ is a strong clique in $G''$.
  \item Every vertex of $G''$ is contained in a strong clique.
  \item Every clique from the following collection of cliques in $G''$ is strong: $$\{C\}\cup \{\{w,w'\}\mid w\in V(G')\setminus C\}\,.$$
  \item The vertex set of $G''$ can be partitioned into strong cliques.
\end{enumerate}
\end{restatable}

%\propreductions*
\begin{proof}
First we show that $G'$ is diamond-free, or, equivalently, that every edge of $G'$ is in a unique maximal clique (Lemma \ref{lem:maximal-cliques-diamond-free}). Consider an edge $e$ of $G'$ and let $(u,i)$ and $(v,j)$ be its endpoints. Then either $u = v$ or $i = j$. If $u = v$, then $C_u$ is the only maximal clique of $G'$ containing
$e$. If $i = j = 0$, then $C$ is the only maximal clique of $G'$ containing $e$. If $i = j \in \{1,2,3\}$, then $uv\in E(G)$ and since $G$ is triangle-free, $e$ itself is a maximal clique in $G'$.
We infer that $G'$ is diamond-free. Similarly, it can be seen that $G''$, too, is diamond-free.

We now prove the implication $1\Rightarrow 2$. Suppose that $C$ is not strong in $G'$. Then, by Lemma \ref{lem:strong-cliques}, there is a stable set $S'\subseteq V(G')\setminus C$ dominating $C$. Since $S'$ dominates $C$, for each $v\in V(G)$ there exists some $i\in \{1,2,3\}$ such that $(v,i)\in S'$. Furthermore, since the set of all vertices in $G'$ having the same first coordinate is a clique and $S'$ is a stable set in $G'$, we infer that $S'$ cannot contain two vertices with the same first coordinate. This means that for each $v\in V(G)$ there is a unique $i\in \{1,2,3\}$ such that $(v,i)\in S'$. For each $i\in \{1,2,3\}$, let $S_i=\{v\in V(G)\mid (v,i)\in S'\}$.
By the construction of $G'$, each $S_i$ is a stable set in $G$.
Moreover, due to the above uniqueness property, $\{S_1,S_2,S_3\}$ is a partition of $V(G)$ into three stable sets, implying that $G$ is 3-colorable.

Next we prove the implication $2\Rightarrow 1$. Suppose that $G$ is $3$-colorable and let $\{S_1,S_2,S_3\}$ be a partition of $V(G)$ into three stable sets. For $i\in \{1,2,3\}$, let $S_i' = S_i\times\{i\}$. Then $S = S_1'\cup S_2'\cup S_3'$ is a stable set in $G'-C$ dominating $C$. By Lemma \ref{lem:strong-cliques}, $C$ is not strong in $G'$.
See Figure~\ref{fig:reduction-2}.

\begin{figure}[h]
  \centering
  \includegraphics[scale=0.8]{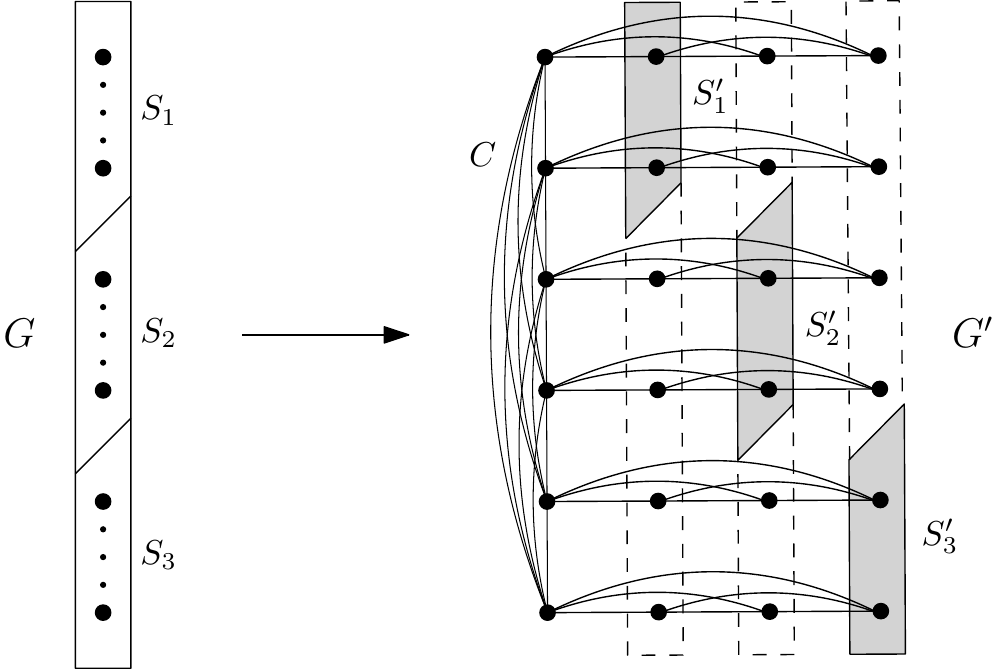}
  \caption{Mapping a $3$-coloring of $G$ to a stable set in $G'-C$ dominating $C$.}
  \label{fig:reduction-2}
\end{figure}

The implication $2\Rightarrow 3$ is immediate. Next we prove the implication $3\Rightarrow 2$. For this, we show that no maximal clique $C'$ in $G'$ other than $C$ is  strong. Let $C'\neq C$ be a maximal clique in $G'$. Since $G$ is triangle-free it follows that either $C'= \{(u,i),(v,i)\}$ for some $uv\in E(G)$ and $i\in\{1,2,3\}$ or $C'=C_v$ for some $v\in V(G)$. In the first case, suppose without loss of generality that $i=1$, then $C'=\{(u,1),(v,1)\}$ for some $uv\in E(G)$. Then, $\{(u,2),(v,3)\}$ is a stable set in $G'-C'$ dominating $C'$. Hence, $C'$ is not strong. Now, let $C'=C_v$ for some $v\in V(G)$. Let $u_1,u_2,u_3$ be three different neighbors of $v$ in $G$  and let $u_0\in V(G)\setminus \{v,u_1,u_2,u_3\}$.
(Recall that such vertices exist since $G\in \mathcal{G}$.) Then, $\{(u_i,i)\mid i\in \{0,1,2,3\}\}$ is a stable set in $G'-C'$ dominating $C'$. Hence, $C'$ is not strong.

Next we prove that statements 2 and 4 are equivalent. By Lemma \ref{lem:strong-cliques}, it suffices to show that there is a stable set in $G'-C$ dominating $C$ if and only if there is a stable set in $G''-C$  dominating $C$. But these two conditions are clearly equivalent, since
$G'[N_{G'}[C]] = G''[N_{G''}[C]] =G'$.

Before proving the remaining implications, let us note that for any vertex $v\in V(G)$ and $i\in \{1,2,3\}$, the clique $\{(v,i),(v,i)'\}$
is a simplicial clique in $G''$. In particular, $\{\{w,w'\}\mid w\in V(G')\setminus C\}$ is a collection of strong cliques in $G''$.
This immediately implies that statements 4 and 6 are equivalent.

Clearly, statement 6 implies statement 7 and statement 7 implies statement 5. Finally, we prove the implication $5\Rightarrow 4$. Suppose that every vertex of $G''$ is contained in a strong clique. Note that for every vertex $v\in V(G)$, the sets $C_v$ and $C$ are the only maximal cliques in $G''$ containing vertex $(v,0)$. Observe that $\{(v,i)'\mid i\in\{1,2,3\}\}\cup\{(u,0)\}$ where $u\in V(G)\setminus\{v\}$ is a stable set in $G''-C_v$ dominating $C_v$. By Lemma \ref{lem:strong-cliques}, $C_v$ is not strong in $G''$. Thus, since $(v,0)$ is contained in a strong clique, we infer that $C$ is a strong clique in $G''$.
\qed\end{proof}

Using Proposition~\ref{prop:reductions} we derive the following hardness results.

\thmHard*
\begin{sloppypar}
\begin{proof}
First we establish membership in {\sf co-NP} of the first four problems.

{\sc Strong Clique} is in {\sf co-NP} when restricted to any class of graphs. Given a \noinstance~to the problem, consisting of a graph $G$ and a clique $C$ in $G$, a short certificate of the fact that $C$ is not a strong clique in $G$ is given by any stable set $S$ in $G-C$ that dominates $C$ (such a stable set exists by Lemma~\ref{lem:strong-cliques}).

{\sc Strong Clique Existence} is in {\sf co-NP} when restricted to the class of diamond-free graphs. Given a \noinstance~to the problem, consisting of a diamond-free graph $G$ in which no clique is strong, a short certificate of this fact is a collection of  stable sets $S_C$, one for each maximal clique $C$ in $G$, such that each $S_C$ is a stable set in $G-C$ dominating $C$. Such a collection of stable sets exists by Lemma~\ref{lem:strong-cliques}, and it is of polynomial size by Corollary~\ref{cor:maximal-cliques-diamond-free}.

{\sc Strong Clique Vertex Cover} is in {\sf co-NP} when restricted to the class of diamond-free graphs. Consider a \noinstance~to the {\sc Strong Clique Vertex Cover} problem, consisting of a diamond-free graph $G$ in which not all vertices are contained in a strong clique. In this case, a short certificate consists of a vertex $v\in V(G)$ that is not contained in any strong clique and a collection of stable sets $S_C$, one for each maximal clique $C$ containing $v$, such that $S_C$ is a stable set in $G-C$ dominating $C$. Such a collection of stable sets exists by Lemma~\ref{lem:strong-cliques}, and it is of polynomial size by Corollary~\ref{cor:maximal-cliques-diamond-free}.

{\sc Strong Clique Partition} is in {\sf co-NP} in any class of graphs. A \noinstance~to the {\sc Strong Clique Partition} problem consists of a graph $G$ and a partition of its vertex set into cliques such that not all cliques in the partition are strong. A short certificate of the fact that this is indeed a \noinstance~consists of a clique $C$ from the given collection that is not strong and a stable set in $G-C$ dominating $C$. (Such a stable set exists by Lemma~\ref{lem:strong-cliques}.)

We prove hardness of all the five problems in the class of diamond-free graphs using a reduction from the {\sc $3$-Colorability} problem in the class $\mathcal{G}$, which is {\sf NP}-complete by Theorem~\ref{thm:3-col}. Note that by Proposition~\ref{prop:reductions}, the derived graphs $G'$ and $G''$ are diamond-free. Clearly, they can be computed in polynomial time from $G$.
Hardness of the {\sc Strong Clique} problem follows from the fact that $G$ is $3$-colorable if and only if $C$ is not a strong clique in $G'$. Hardness of the {\sc Strong Clique Existence} problem follows from the fact that $G$ is $3$-colorable if and only if $G'$ does not have any strong cliques. Hardness of the {\sc Strong Clique Vertex Cover} problem follows from the fact that $G$ is $3$-colorable if and only if not every vertex of $G''$ is contained in a strong clique.
Consider now the {\sc Strong Clique Partition} problem. By Proposition~\ref{prop:reductions}, $G$ is $3$-colorable if and only if not every clique from the following collection of cliques in $G''$ is strong: $\{C\}\cup \{\{w,w'\}\mid w\in V(G')\setminus C\}\,.$ Note that the cliques in this collection form a partition of the vertex set of $G''$ and that they can be computed in polynomial time from $G$. Hardness of the {\sc Strong Clique Partition Existence} problem follows from the fact that $G$ is $3$-colorable if and only if the vertex set of $G''$ cannot be partitioned into strong cliques.
\qed\end{proof}
\end{sloppypar}

\section{Conclusion}

\begin{sloppypar}
We established the complexity of seven problems related to strong cliques in the class of diamond-free graphs. Five of these problems remain intractable and the remaining two become solvable in linear time.
Our work refines the boundaries of known areas of tractability and intractability of algorithmic problems related to strong cliques in graphs. Besides the open problems of the complexity of testing whether every maximal clique is strong, or whether every edge is contained in a strong clique, many other interesting questions remain. For example, it is still open whether there exists a polynomial-time algorithm to recognize the class of strongly perfect graphs, introduced in 1984 by Berge and Duchet~\cite{MR778749} and defined as graphs in which every induced subgraph has a strong stable set. To the best of our knowledge, the recognition complexity of strongly perfect graphs is also open when restricted to the class of diamond-free graphs.
\end{sloppypar}

%For example, it is still open whether there exists a polynomial-time algorithm to recognize the class of strongly perfect graphs, defined as graphs in which the complement of every induced subgraph has a strong clique.
%For example, it would be natural to ask which diamond-free graphs contain a strong clique in each induced subgraph, and the same question for the complementary graphs.
%Both of these questions are a special case of the open problem about the existence of a polynomial-time recognition algorithm for the class of strongly perfect graphs introduced in 1984 by Berge and Duchet~\cite{MR778749}.

\subsection*{Acknowledgements}

\begin{sloppypar}
The authors are grateful to Ademir Hujdurovi\'c for helpful discussions and the anonymous reviewers for constructive remarks. The second named author has been supported by the European Union's Horizon 2020 research and innovation programme under the Marie Sklodowska-Curie grant agreement No.~734922, by Mexico's CONACYT scholarship 254379/438356, by Erasmus+ for practices (SMT) Action KA103 - Project 2017/2019, by MICINN from the Spanish Government under project PGC2018-095471-B-I00, and by AGAUR from the Catalan Government under project 2017SGR1087. The first and third named authors are supported in part by the Slovenian Research Agency (I0-0035, research programs P1-0285 and P1-0404, and research projects J1-9110, N1-0102). Part of this work was done while the third named author was visiting LAMSADE, University Paris-Dauphine; their support and hospitality is gratefully acknowledged.
\end{sloppypar}

\bibliographystyle{abbrv}
\bibliography{biblio}

\end{document}